\documentclass[oneside, 11pt,reqno,tbtags]{amsart}

\usepackage{amsthm,amssymb,amsfonts,amsmath}
\usepackage{amscd} 
\usepackage{mathrsfs}
\usepackage{appendix}
\usepackage[numbers, sort&compress]{natbib} 
\usepackage{graphicx}
\usepackage{vmargin}
\usepackage{setspace}
\usepackage[colorlinks=true, citecolor=cyan]{hyperref}
\usepackage{xspace}
\usepackage{array}
\usepackage[dvipsnames]{xcolor}
\usepackage{paralist}
\usepackage{color}
\usepackage[normalem]{ulem}
\usepackage{stmaryrd} 
\usepackage[refpage,noprefix,intoc]{nomencl} 
\usepackage{tcolorbox}
\usepackage{bbm,mathtools,tikz}
\usepackage{xcolor}
\usepackage{MnSymbol,wasysym}

\providecommand{\Z}{}
\providecommand{\N}{}

\renewcommand{\Z}{\mathbb{Z}}
\renewcommand{\N}{{\mathbb N}}

\renewcommand{\P}{\mathbb{P}}

\DeclareMathAlphabet{\mathmybb}{U}{bbold}{m}{n}
\newcommand{\one}{\mathmybb{1}}

\newcommand\cD{\mathcal D}

\newcommand\cP{\mathcal P}

\newcommand\cT{{\mathcal T}}

\usepackage{crossreftools,booktabs}

\makeatletter
\newcommand{\optionaldesc}[2]{%
  \phantomsection
  #1\protected@edef\@currentlabel{#1}\label{#2}%
}
\makeatother

\newcommand{\rd}{\mathfrak{d}}
\newcommand{\rp}{\mathfrak{p}}
\renewcommand{\rm}{\mathfrak{m}}

\newcommand{\rt}{\mathfrak{t}}
\renewcommand{\rq}{\mathfrak{q}}

\newcommand{\BijPos}{\ensuremath{\varphi}}
\newcommand{\BijTree}{\ensuremath{\psi}}

\providecommand{\ora}[1]{}
\renewcommand{\ora}[1]{\overrightarrow{#1}}
\DeclareRobustCommand{\SkipTocEntry}[5]{} 

\newtheorem{theo}{Theorem}
\newtheorem{lemm}[theo]{Lemma}
\newtheorem{prop}[theo]{Proposition}

\newtheorem{defi}[theo]{Definition}
\newtheorem{rema}[theo]{Remark}

\makenomenclature										
\graphicspath{ {./images/} }
\usepackage{wrapfig}

\numberwithin{equation}{section}
\numberwithin{theo}{section}

\begin{document}
\date{} 

\title{A tree bijection for Eulerian orientations on quartic maps}

\author{Marie Albenque}
\address{CNRS -- Université Paris Cité, Paris, France.}
\email{malbenque@irif.fr}

\author{Enrica Duchi}
\address{Université Paris Cité, Paris, France}
\email{duchi@irif.fr}

\author{Juliette Schabanel}
\address{LaBRI -- Université de Bordeaux, Bordeaux, France}
\email{jschaban@phare.normalesup.org}

\subjclass[2010]{X} 

\begin{abstract} 
In this paper we construct a recursive bijection between $4$-valent maps endowed with a Eulerian orientation -- also known as the ice model --  and a family of unary-binary trees. This answers a question of Bousquet-Mélou and Elvey Price, who established that both these families have the same generating function in \cite{MBM_AEP_25}.

\end{abstract} 

\maketitle

\section{Introduction}

Since the work of Tutte in the sixties~\cite{Tutte62, Tutte63}, the enumeration of planar maps (that are graphs embedded in the sphere) has grown into a classical and active area of research. Using a method based on recursive decompositions, Tutte obtained many closed enumerative formulas for various families of maps characterized by degree or girth constraints. The striking simplicity of these formulas called for a direct combinatorial interpretation, and starting from the pioneering work of Schaeffer~\cite{Schaeffer97, PhDSchaeffer}, many bijections have been constructed between some families of decorated trees and planar maps~\cite{BDG02, BDG04, PS06, BF12}. 
These bijections not only give a satisfactory explanation of why the enumerative formulas are so simple, but they were also instrumental in the study of the scaling limit of random planar maps and in the definition of the Brownian sphere, as a universal limiting object for many families of maps~\cite{Miermont, LeGall}.

Maps have also appeared independently quite early in the physics literature~\cite{tHooft,BIPZ}, through the study of matrix integrals, and then as a discrete approximation of the 2-dimensional quantum gravity. In the latter point of view, it makes sense to study \emph{decorated maps}, which are maps with an additional structure coming either from a statistical physics perspective, such as the Ising model~\cite{BoulatovKazakov}, or from a more graph-theoretic perspective, such as maps with a coloring of their vertices~\cite{TutteColorTriang}, or with a distinguished spanning tree~\cite{Mullin}.

In this paper, we consider maps equipped with a \emph{Eulerian orientation}, that is an orientation of its edges such that the indegree and the outdegree of each vertex are equal. More precisely, we focus on maps where all vertices have degree 4 -- such maps are called quartic -- and in this setting a Eulerian orientation is the counterpart for maps of the ice model defined on the square lattice by Pauling in 1935. In the combinatorics literature, this model was first studied by Bonichon et al~\cite{Bonichon}, who obtained several bounds for the number of maps equipped with a Eulerian orientation. It was subsequently studied in~\cite{AEP_TG}, where the authors computed many initial coefficients and, from these data, conjectured an unusual  asymptotic behaviour for the enumerative sequence. The enumeration problem was eventually solved in~\cite{MBM_AEP_20}. In a later work~\cite{MBM_AEP_25}, the authors refined these results and conjectured the existence of a bijection between quartic Eulerian orientations and a family of suitably charged unary-binary trees. 

We prove this conjecture by constructing a recursive bijection between labeled quadrangulations, which are dual to quartic Eulerian orientations (see Figure \ref{fig:LQvsQO}), and this family of trees. This bijection is similar to Schaeffer's decomposition trees \cite{PhDSchaeffer}, which is a bijection between (unoriented) planar maps and labeled unary-binary trees. It coincides with our bijection when restricted to unlabeled quadrangulations.
Aside from providing a combinatorial explanation to some equations established in~\cite{MBM_AEP_25}, our bijection also opens the door to efficient sampling of labeled quadrangulations through Boltzmann samplers (see Section~\ref{sec:discuss}).

\begin{figure}[ht]
  \centering 
  \includegraphics[height=4.2cm,page=12]{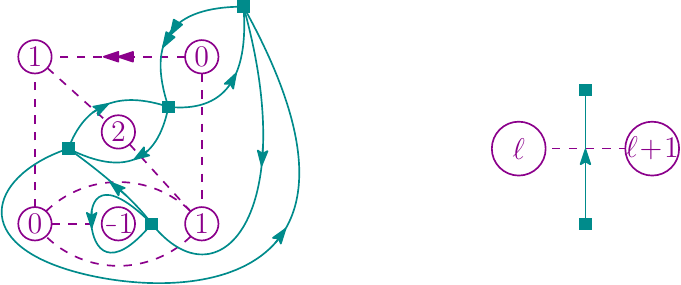}
  \caption{A quartic Eulerian orientation (in teal plain edges) and its dual labeled quadrangulation (in purple dashed edges). The labeling rule is shown on the right.}
  \label{fig:LQvsQO}
\end{figure}

It is natural to ask whether we could obtain a non-recursive bijection; and indeed, we have a way to construct directly the tree by exploring the faces of the quadrangulation. It should also be possible to describe a process to build the map starting from the tree. However, both constructions remain recursive in essence: the proofs that they are indeed inverse bijections still rely on recursive arguments. Thus, the search for a genuinely non-recursive bijection remains open; this point is also discussed in Section~\ref{sub:leaves}.

Let us conclude this introduction by mentioning some recent related works, and how they relate to our result. In~\cite{MBM_AEP_20}, Bousquet-Mélou and Elvey Price also enumerated quartic Eulerian orientation without alternating vertices, and conjectured that this subclass is in bijection with a family of decorated binary trees. Note that this subclass corresponds by duality to so-called \emph{colorful quadrangulations}, which have been studied recently in the mathematical physics literature~\cite{Budd25}, motivated by some connections with \emph{Jackiw-Teitelboim gravity}. In~\cite{Zon25}, Zonneveld partially answered the conjecture of~\cite{MBM_AEP_20} by constructing a bijection between colorful quadrangulations and a family of trees.  However, there is no clear bijection between his trees and the simpler model proposed by Bousquet-Mélou and Elvey Price. Note that the restriction of our bijection to colorful quadrangulations gives yet another family of trees, which is closer to the family described in~\cite{MBM_AEP_20}, but still without any explicit bijection.

\textbf{Outline of the paper.} In Section~\ref{sec:def}, we begin with definitions on maps, orientations and labelings. In Section~\ref{sec:positive}, we recall a bijection between unlabeled quadrangulations and some unary-binary trees, which we generalize to labeled quadrangulations in Section~\ref{sec:bij recursive}. The approach will be to recursively decompose the labeled map, and encode the decomposition into unary-binary tree. Finally, in Section~\ref{sec:discuss}, we explain how our bijection allows us to recover more directly some previously known equations on labeled quadrangulations and how we can use it to sample labeled quadrangulations; and a second subsection is then dedicated to a discussion about statistics. We conclude the paper with an appendix in which we describe a procedure to directly obtain the tree by exploring the faces of the quadrangulation. 

\textbf{Acknowledgments} 
The authors thank Mireille Bousquet-Mélou and Andrew Elvey-Price for many insightful discussions about this problem, which they first raise in an open problem session in the kickoff meeting of the ANR grant CartesEtPlus. The authors also thank Dominique Poulalhon for many inspiring discussions on this problem.

This work was partially funded by the ANR grants IsoMa (ANR-21-CE48-0007) and CartesEtPlus (ANR-23-CE48-0018).

\section{Definitions}
\label{sec:def}

\subsection{Unary-binary trees and well-charged trees}

A \emph{unary-binary} tree $\rt$ is a tree where all vertices have either 0, 1 or 2 children. Vertices with $0$ children are called \emph{leaves} and the other vertices \emph{inner vertices}. We denote by $V(\rt)$ the set of inner vertices of $\rt$. The \emph{size} $|\rt|$ of $\rt$ is defined as its number of leaves minus one (or equivalently as its number of binary inner vertices).

The \emph{charge} $c_\rt$ of a unary-binary tree $\rt$ is the difference between its number of binary vertices and its number of unary vertices. Moreover, for $u\in V(\rt)$, write $\rt_u$ for the subtree of $\rt$ rooted at $u$. Then, the charge of $\rt$ at $u$ -- denoted by $c_\rt(u)$, or $c(u)$ if $\rt$ is clear from the context -- is defined as $c_{\rt_u}$.

A unary-binary tree $\rt$ is said to be \emph{well-charged} if for all $u \in V(\rt)$, $c_\rt(u) \neq 0$. For $k\in \mathbb{Z}$, we define $\cT^{(k)}$ as the family of well-charged trees with root vertex of charge $k$. Observe that if such a tree has $n+1$ leaves then it must have $n$ binary vertices and $n-k$ unary vertices. By convention, the only tree of $\cT^{(0)}$ is the tree reduced to a single leaf.

A leaf is said to be \emph{left} (resp. \emph{right}, resp.\emph{unary}) if it is the left child of a binary vertex (resp. right child, resp. child of a unary vertex).

\begin{figure}[ht]
  \centering 
  \includegraphics[height = 4cm,page=1]{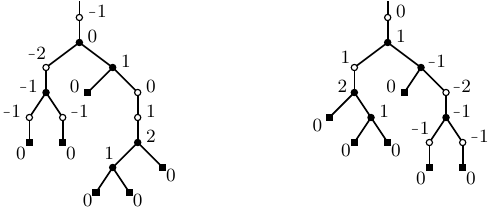}
  \label{fig:UnBinTree}
  \caption{Two unary-binary trees of size $5$. The labels indicate the charges at each vertex. Left: A unary-binary (not well-charged). Right: A well0charged tree with $3$ left leaves, $1$ right leaf and $2$ unary leaves.}
\end{figure}

\subsection{Labeled quadrangulations}

A \emph{planar map} $\rm$ is a proper embedding of a connected planar graph onto the sphere, considered up to orientation preserving homeomorphisms. Loops and multi-edges are allowed. The vertices and edges of the maps are those of the graph and its faces are the connected components of its complement. Their sets are respectively denoted by $V(\rm)$, $E(\rm)$ and $F(\rm)$.
The numbers of vertices, edges and faces of a planar map M, denoted by $v(\rm)$, $e(\rm)$ and $f(\rm)$, are related by Euler’s formula $v(\rm) + f(\rm) = e(\rm) + 2$. The \emph{degree} of a vertex or face is the number of edges it is adjacent to, counted with multiplicity. A \emph{corner} is an angular sector formed by two consecutive edges around a vertex, we denote by $C(\rm)$ the set of corners.

All our maps will be \emph{rooted}, that is they have a distinguished oriented edge or equivalently a distinguished corner. The tail and the head of the root edge are respectively called the \emph{root vertex} and the \emph{co-root vertex}. The face to the right (resp. to the left) of the root edge is called the \emph{root face} (resp. the \emph{co-root face}). In all our figures, the root face will be represented as the outer infinite face. 

A \emph{quadrangulation} is a map where all the faces have degree 4. A \emph{quasi-quadrangulation} is a map where all faces except the root face have degree $4$, observe that the root face of a quasi-quadrangulation must have an even degree.

A \emph{labeled map} $(\rm,\ell)$ is a map together with a labeling $\ell$ of its vertices, such that: 
\vspace{-0.1cm}
\begin{itemize}
\item For any $v\in V(\rm)$, $\ell(v)\in \mathbb{Z}$,
\item For any $u,v \in V(\rm)$ such that $\{u,v\}\in E(\rm)$, we have $\ell(v)-\ell(u)\in\{-1,1\}$,
\item The root vertex and co-root vertex are respectively labeled by $0$ and $1$. 
\end{itemize}
Given a labeled map $(\rm,\ell)$, we extend the definition of $\ell$ to $C(\rm)$, by defining the label of a corner as the label of its incident vertex. 

Note that the definition of a labeled map readily implies that the underlying map is \emph{bipartite}. This will always be the case in this article, since all the maps considered have only faces of even degrees. 

\begin{rema}
    As mentioned in the introduction, labeled quadrangulations are in bijection with quartic Eulerian orientations via duality. For the sake of completeness, let us recall this construction (see Figure~\ref{fig:LQvsQO}). The \emph{dual} of a quartic Eulerian orientation $\rq$ is the quadrangulation $\rm$ obtained by placing a vertex in each face of $\rq$ and, for each edge $e$ of $\rq$, an edge between the vertices in the two faces it separates. The root edge of $\rm$ is then the edge corresponding to the root edge of $\rq$, oriented from its left to its right. The labeling is then obtained by assigning label $0$ to the root vertex of $\rm$, and propagating the labeling in such a way that for each edge of $\rm$, the endpoint with the larger label lies on the right of the oriented edge of $\rq$ it crosses. The Eulerian condition on the orientation of $\rq$ guarantees that this labeling is globally well defined. Observe that, since quadrangulations are bipartite, there are only two possible label patterns for a face of the quadrangulation: $(i, i+1, i+2, i+1)$ and $(i, i+1, i, i+1)$, which correspond to the two possible orientations around the associated vertex of $\rd$: out-out-in-in or out-in-out-in. Conversely, the labeling determines the orientation uniquely, and by the observation above the orientation obtained is necessarily Eulerian.
\end{rema}

\subsection{Result}

In~\cite{MBM_AEP_20}, Bousquet-Mélou and Elvey Price established in a purely computational way that the number of elements of $\cT^{(1)}$ of size $n+1$ is equal to the number of labeled quadrangulations with $n$ faces. The main result of this article is to give a bijective proof of this result, namely we prove:

\begin{theo}\label{thm:main}
There exists an explicit (recursive) bijective correspondence between elements of $\cT^{(1)}$ with size $n+1$ and the set of labeled quadrangulations with $n$ faces.
\end{theo}

We in fact prove a stronger result by giving a family of maps in bijection with each of the $\cT^{(k)}$ for $k \in \Z$.

\section{Warm-Up: Fully positive trees and Unlabeled quadrangulation}
\label{sec:positive}

As a warm-up before presenting the full bijection, we start with the simpler setting of unlabeled quadrangulations. These turn out to correspond to a natural subclass of unary-binary trees, which we call \emph{fully positive trees}. The idea of our bijection in this case is to give a ``à la Tutte" recursive decomposition, which is very similar to some construction given by Schaeffer~\cite{PhDSchaeffer}; we revisit it here as a blueprint for the general bijection. 

Note that unlabeled quadrangulations can be seen as a special case of labeled ones, by equipping each map with its unique valid labeling with only $0$ and $1$ as authorized labels. With this identification, the bijection presented in this section will be a particular case of the general one.

\subsection{Fully positive trees}
A unary-binary tree $\rt$ is said to be \emph{fully-positive} if for any $u\in V(\rt)$, $c(u)>0$. Fully-positive trees admit the following equivalent characterization: 

\begin{prop}
\label{prop:pos-leaf}
    For any $k>0$, a tree $\rt \in \cT^{(k)}$ is fully positive if and only if it has no unary leaf, i.e all its leaves are binary.
\end{prop}
\begin{proof}
    Let $\rt \in \cT^{(k)}$. If $\rt$ has a unary leaf, then its parent has charge $-1$ and hence $\rt$ is not fully positive.
    
    Conversely, assume $\rt$ has an inner vertex of negative charge and let $u \in V(\rt)$ be such a vertex, with the additional assumption that there is no other inner vertex of negative charge in $\rt_u$. If $u$ were binary, its charge would be $1 + c(u_1) + c(u_2)$ (where we write $u_1$ and $u_2$ for its children) which is positive by assumption on $u$. So $u$ is unary with child $v$, and $c(u) = c(v) -1 <0$. This implies that $c(v) = 0$, and hence $v$ is necessarily a leaf. 
\end{proof}

\begin{theo}
\label{thm:pos-unlabel}
   For all $k, n \geqslant 0$, there is an explicit bijection $\BijPos$ between fully positive trees of charge $k$ with $2n-k$ inner vertices (and so $n+1$ leaves) and (unlabeled) quasi-quadrangulations with $n-k$ quadrangles and boundary of length $2k$ (and so $n+1$ vertices).

    This bijection moreover preserves the following statistics: 
    
    \centering \begin{tabular}{ccc}
         Fully Positive Trees &$\longleftrightarrow$& Quasi-Quadrangulation  \\
         \hline
         Inner vertices &$\longleftrightarrow$& Edges \\
         Unary vertices &$\longleftrightarrow$& Faces \\
         Left leaves &$\longleftrightarrow$& Vertices labeled $1$\\
         Right leaves &$\longleftrightarrow$& Vertices labeled $0$ \\
    \end{tabular}
\end{theo}

\subsection{Proof of Theorem~\ref{thm:pos-unlabel}}
\label{sub:bij pos}

The idea of the proof is to encode the classical Tutte's decomposition of quasi-quadrangulations into unary-binary trees. In the decomposition of quadrangulations, deleting the root $e$ which either gives two smaller maps (when $e$ is a bridge) or a single map whose boundary has two more edges (otherwise). Those two cases will be respectively translated in the tree setting as binary vertices, which have two subtrees, and unary vertices whose unique subtree has charge $1$ higher.

Formally, we define the bijection $\BijTree$ that maps a quadrangulation $\rm$ to a tree $\BijTree(\rm)$ as follows,  see Figure~\ref{fig:corresp positive}:
\begin{enumerate}
    \item If $\rm$ is reduced to a single vertex, $\BijTree(\rm)$ is reduced to a leaf.
    \item 
    If the root edge of $\rm$ is a bridge, let $\rm_0$ be the component containing the root vertex, let $\rm_1$ be the other component, and let $\BijTree(\rm)$ be the tree consisting of a binary root node whose right subtree is $\BijTree(\rm_0)$ and left subtree is $\BijTree(\rm_1)$.
    \item If the root edge is not a bridge, let $\rm'$ be the map obtained by deleting it and let $\BijTree(\rm)$ be the tree obtained by adding a unary vertex to $\BijTree(\rm')$.
\end{enumerate}
Observe that each edge of $\rm$ produces exactly one inner vertex of $\BijTree(\rm)$, and each face produces one unary vertex (when applying construction $(3)$). Moreover, if $\rm$ is not reduced to a single vertex, the tree $\BijTree(\rm)$ has only binary leaves by construction. Indeed, the subtree of a unary vertex is the image of a map with boundary length at least $4$, so it cannot be reduced to a leaf. Hence, by Proposition~\ref{prop:pos-leaf}, the tree $\BijTree(\rm)$ is fully positive for each quasi-quadrangulation $\rm$. Finally, since the root edge is always oriented from $0$ to $1$, $0$-vertices are always mapped to right leaves and $1$-vertices to left leaves (because of the choice in construction $(2)$).

\begin{figure}[ht]
    \centering
    \includegraphics[page=2, height=4.5cm]{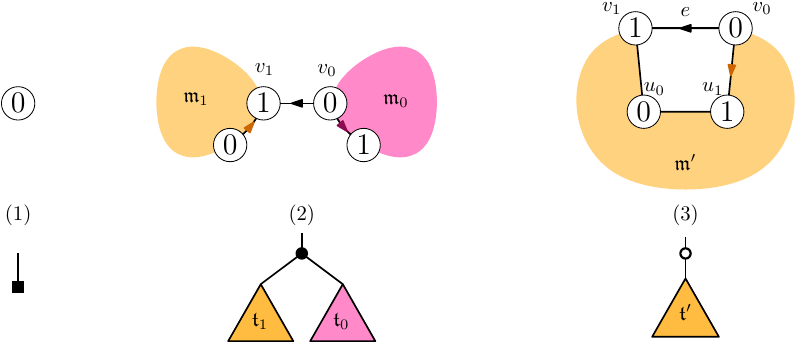}
    \caption{The parallel decomposition on fully positive trees and on unlabeled quadrangulations.}
    \label{fig:corresp positive}
\end{figure}

To prove that this construction is bijective, we exhibit the inverse construction $\BijPos$, i.e. the function such that $\BijPos(\BijTree(\rm))=\rm$. It can be constructed naturally by reversing the different cases in the construction of $\BijTree(\rm)$, although some care is needed to ensure that the labels are handled correctly. Formally, $\BijPos$ is defined recursively as follows: 
    \begin{enumerate}
    \item For a right (resp. left) leaf $\ell$, $\BijPos(\ell)$ is a single vertex labeled $0$ (resp. $1$).
    \item If $u$ is a binary vertex with children $u_0$ and $u_1$, $\BijPos(\rt_u)$ is obtained by linking $\BijPos(\rt_{u_0})$ and $\BijPos(\rt_{u_1})$ with an edge going from the root vertex of $\BijPos(\rt_{u_1})$ to the co-root vertex of $\BijPos(\rt_{u_0})$, and by rooting the resulting map on this edge oriented from $0$ to $1$.
    \item If $u$ is a unary vertex with child $v$, then $\BijPos(\rt_u)$ is obtained from $\BijPos(\rt_v)$ by adding an edge from its root corner to the third corner along the boundary (closing a face of degree~$4$).
\end{enumerate}
Observe that the construction in the last item can be performed if and only if the map has boundary length at least $4$. Since the tree is fully positive, and keeping the same notations as above, $c_\rt(u)\geq 1$, and hence we have $c_\rt(v)=c_\rt(u)+1 \geq 2$. To prove that the construction is well defined, it only remains to prove that the charge of $\BijTree(\rm)$ does indeed equal half the boundary length of $\rm$.

We prove it by induction: 
\begin{enumerate}
    \item A leaf has charge $0$ and a single vertex has boundary of length $0$.
    \item If $c(u_0) = k_0$ and $c(u_1) = k_1$ then $c(u) = k_0 + k_1 +1$ and the boundary of $\BijPos(\rt_u)$ has length $2k_0+2k_1+2$ by construction. 
    \item If $c(v) = k$, then $c(u)=k-1$ and the boundary of $\BijPos(\rt_u)$ has length $2k - 3 + 1 = 2(k-1)$. 
\end{enumerate}

Finally, we need to check that the quasi-quadrangulation obtained is properly labeled. It is clear that binary vertices (construction $(2)$) only creates edges between vertices with different labels. For unary vertices (construction $(3)$), since $\BijPos(\rt_v)$ is properly labeled by the induction hypothesis, the labels of the corners of its root face are alternating. This implies that the vertex incident to the third corner after the root corner (which is labeled $0$) is labeled $1$, hence the new edge is also between vertices with different labels.
This concludes the proof.

\section{The recursive correspondence}
\label{sec:bij recursive}

In this section, we extend the bijection of Theorem~\ref{thm:pos-unlabel} beyond the
fully positive setting, in order to treat well-charged trees of arbitrary
charge. 

As in the fully positive case, trees with positive charge are encoded by
labeled maps with a boundary: more precisely, by the \emph{patches}
introduced in~\cite{MBM_AEP_20}, whose boundary length records the
charge. For trees with negative charge, however, this boundary
interpretation is no longer appropriate. Instead, we use a second family
of labeled almost-quadrangulations, also introduced in~\cite{MBM_AEP_20},
called D-patches. In these maps, digons incident to the root vertex are
allowed, and the opposite of the charge of the tree is recorded by the number of
digons of the associated map. At an informal level, these digons can be thought of as a
``negative boundary'' folded inside the quadrangulation.

We recall the definitions of patches and D-patches in
Section~\ref{sec:patches}. This allows us to state
Theorem~\ref{thm:bijRec}, which extends Theorem~\ref{thm:pos-unlabel} to
all well-charged trees. The remainder of the section is devoted to the
proof, based on parallel recursive decompositions of the two classes of
maps and of the corresponding trees.

\subsection{Patches}
\label{sec:patches}
Following the terminology of~\cite{MBM_AEP_20}, we introduce patches and D-patches (see Figure~\ref{fig:Patches}):
\begin{defi}
A \emph{patch} is a labeled quasi-quadrangulation in which the vertices incident to the root face are alternately labeled $0$ and $1$, with the root vertex labeled $0$. For $k \geqslant 1$, we denote by $\cP^{(k)}$ the set of patches with a boundary of length $2k$.

A \emph{D-patch} is a labeled map in which the root face has degree $2$ and each inner face has degree either $2$ or $4$, those of degree $2$ being incident to the root vertex. We also require that all neighbours of the root vertex are labeled $1$. For $k \geqslant 1$, we denote by $\cD^{(k)}$ the set of D-patches with $k$ digons (including the root face).
\end{defi}

\begin{figure}[ht]
  \centering 
  \includegraphics[height = 5cm,page=3]{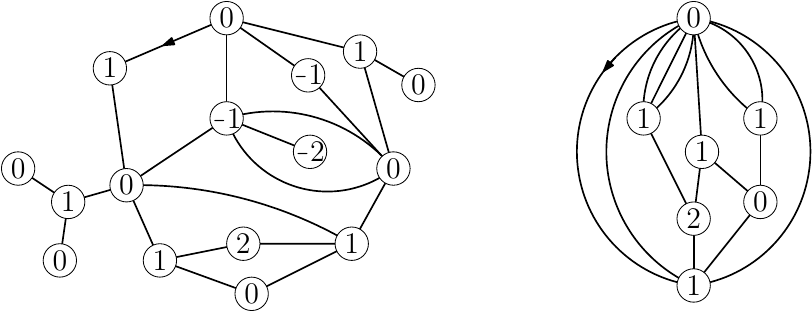}
  \caption{A patch (left) and a D-patch (right).}
  \label{fig:Patches}
\end{figure}

We can now state our main result : 

\begin{theo}\label{thm:bijRec}
For $k \geqslant 1$, there exists an explicit (recursive) bijective correspondence between :
\begin{itemize}
\item Trees of $\cT^{(k)}$, with $n+1$ leaves,
\item and, patches of outer degree $2k$ with $n+1$ vertices.
\end{itemize}
Moreover, there is an explicit bijective recursive correspondence between: 
\begin{itemize}
\item Trees of $\cT^{(-k)}$, with $n+1$ leaves,
\item and, D-patches with $k$ digons and $n+2$ vertices.
\end{itemize}
Additionally, the bijection induces a correspondence between the edges of the maps and the inner vertices of the trees.
\end{theo}

\begin{rema}
Note that there is a small subtlety when considering $D$-patches with only $1$ digon (which corresponds then necessarily to the root face). Indeed, such a map can either be considered: 
\begin{itemize}
    \item as a $D$-patch with 1 digon and be matched to a tree of $\cT^{(-1)}$ by the bijection, or,
    \item as a patch of outer degree 2 (in which all the vertices incident to the root vertex are labeled 1) and be matched to a tree of $\cT^{(1)}$ by the bijection.
\end{itemize}
\end{rema}

Before moving to the proof, which will be the object of the subsequent subsection, let us explain how this implies Theorem~\ref{thm:main} and provide a combinatorial explanation to the equation established by Bousquet-Mélou and Elvey Price in \cite[Equation~(98)]{MBM_AEP_25}.

\begin{proof}[Proof of Theorem~\ref{thm:main}]
Patches with outer degree 2 are in bijection with general labeled quadrangulations, as can be seen by deleting the edge which is incident to the root face and different from the root edge. Thus, the case $k=1$ of the theorem gives a bijection between elements of $\cT^{(1)}$ with $n+1$ leaves and labeled quadrangulations with $n+1$ vertices, which is exactly the statement of Theorem~\ref{thm:main}. 

\end{proof}

\subsection{The 8 possible types in the decomposition of maps and trees}

The decomposition of trees is classical. Given an element of $\cT^{(k)}$, we remove its root vertex and analyse the subtrees that arise. The analysis proceeds case by case, according to the sign of $k$, the degree of the root vertex, and the sign of the charges of the resulting subtrees.

Namely, we have to consider 8 different types, see Figure~\ref{fig:casesByCases}: 

\begin{enumerate}[I.]
\item The root vertex has one child with charge
\begin{enumerate}[{I}.1)]
\item $k'>1$ (hence $k>0$),
\item $k'<0$ (hence $k<0$).
\end{enumerate}
\item The root vertex has two children, and writing $k_1$ and $k_2$ for the charge of respectively its left and right subtrees, we have: 
\begin{enumerate}[{II}.1)]
\item $k_1,k_2\geqslant0$ (and hence $k>0$),
\item $k_1,k_2<0$ (and hence $k<0$),
\item $k>0$, and $k_1<0$ and $k_2\geqslant0$, (or symmetrically $k_1\geqslant0$ and $k_2<0$),
\item $k<0$, and $k_1<0$ and $k_2\geqslant0$, (or symmetrically $k_1\geqslant0$ and $k_2<0$).
\end{enumerate}
\end{enumerate}

To prove Theorem~\ref{thm:bijRec}, we similarly partition our family of maps into 8 different types:
\begin{enumerate}[{I}.1)]
\item Patches with outer degree $2k$, such that the root edge is not a bridge, and such that the co-root face is labeled $(0,1,0,1)$.
\item D-patches with $k$ digons, such that the co-root face is a digon. 
\end{enumerate}
\begin{figure}[ht]
  \centering 
  \includegraphics[width=\linewidth,page=4]{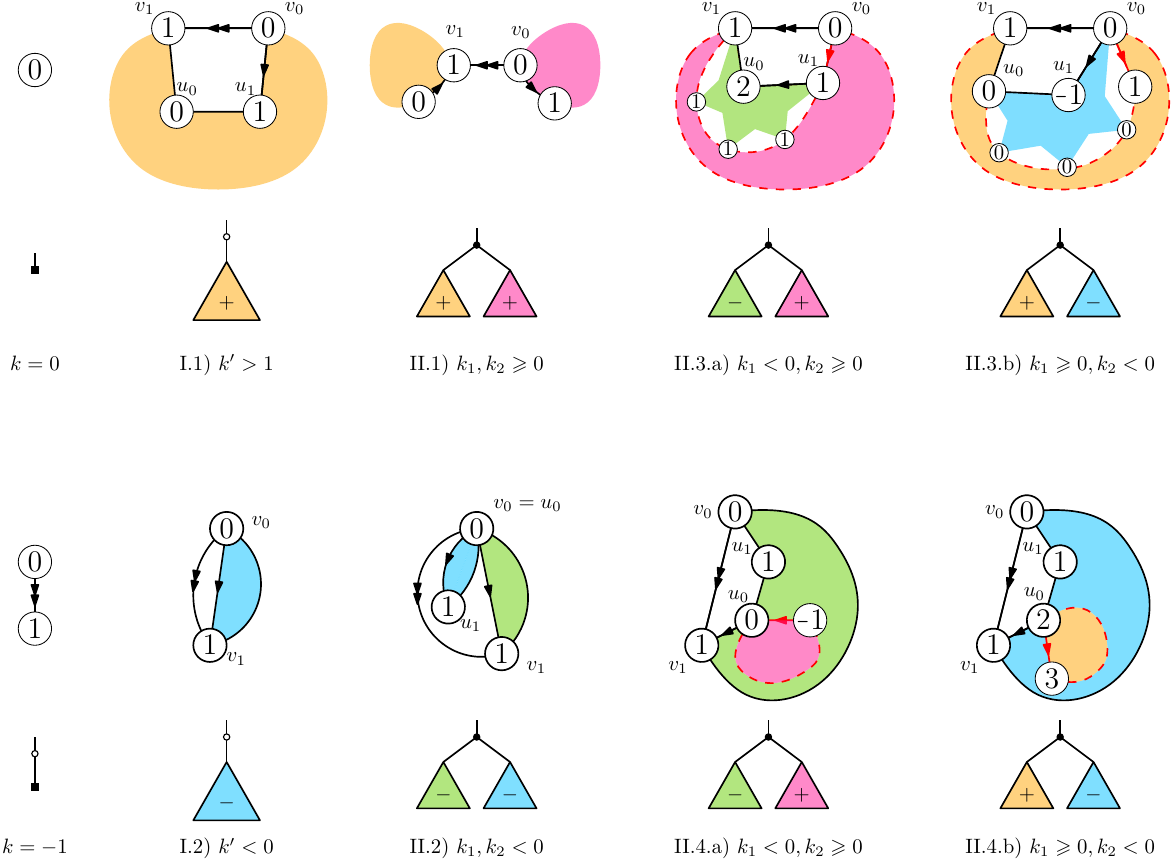}
  \caption{The different types on labeled maps and on trees and their decomposition.}
  \label{fig:casesByCases}
\end{figure}

\begin{enumerate}[{II}.1)]
\item Patches with outer degree $2k$, such that the root edge is a bridge, i.e. the root face and the co-root face are equal.
\item D-patches with $|k|$ digons, such that the co-root face is a quadrangle $(0,1,0,1)$, both $0$s being the root vertex. 
\item Patches with outer degree $2k$, such that the co-root face is labeled $(0,1,2,1)$ (or symmetrically $(0,-1,0,1)$).
\item D-patches with $|k|$ digons, such that the co-root face is labeled $(0,1,0,1)$, the second $0$ being different from the root vertex (or symmetrically $(0,1,2,1)$).
\end{enumerate}
Although writing $|k|$ in place of $k$ may appear cumbersome, this notation will be useful below as well-charged trees with negative charge $k<0$ are in bijection D-patches with $|k|$ digons.

In the rest of this section, we will present a case-by-case recursive decomposition for patches and D-patches, and explain how this decomposition can be encoded by a well-charged tree. Moreover, we make sure that, at each step, the types and the value of $k$ of the maps and of the corresponding trees coincide.

In what follows, given a map $\rm$, we always write $e$ for its root edge and $v_0$ and $v_1$ for its root vertex and its co-root vertex, respectively. When the co-root face is a quadrangle, we denote by $u_0$ its other even-labeled vertex and by $u_1$ its other odd-labeled vertex, with possibly $u_0=v_0$ or $u_1=v_1$.

\subsection{Warm up: Types I}
Consider a labeled map $\rm$ of type I. We delete its root edge, and reroot the resulting map $\rm'$ at the following edge in counterclockwise order around its root vertex, see Figure~\ref{fig:caseI}. This case is encoded into a tree by defining $\rt_\rm$ as the tree obtained by attaching $\rt_{\rm'}$ to a unary vertex. 

If $k>0$, then $\rm'$ is a labeled quadrangulation of the $2(k+1)$-gon. Moreover, since the co-root face of $\rm$ is labeled $(0,1,0,1)$, the sequence of labels around the root face of $\rm'$ reads $0,1,\ldots,0,1$, so $\rm'$ is a patch with one less edge than $\rm$. 

To reverse the operation, one simply needs to add back the edge from the root corner of $\rm'$ to the third next corner of the outer face, so as to create a new inner quadrangle, and root the obtained map on the new edge. Notice that this is not possible only when $\rm'$ has a boundary of length $2$. By the induction hypothesis, it corresponds exactly to forbidding $\rt_{\rm'}$ to be of charge $1$. This is exactly the set of well-charged trees with positive charge, which cannot be grafted to a unary vertex (as this would give a tree of charge $0$).

\medskip

If $k<0$, then $\rm'$ has one digon less than $\rm$, and is then a labeled $(-(k+1))$-patch, with one less edge than $\rm$. 

To reverse the operation, we add back the edge from the root corner of $\rm'$ to the next corner of the outer face so as to create a new inner digon, and we root the obtained map on the new edge. In contrast with the situation above, this is always possible, and, indeed, a well-charged tree with a negative charge can always be grafted to a unary vertex.
\begin{figure}[ht]
  \centering 
  \includegraphics[width=\textwidth,page=5]{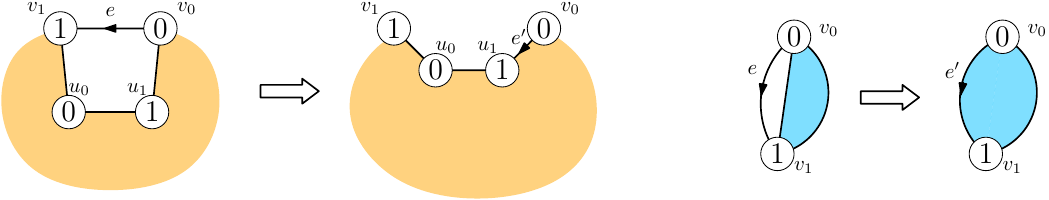}
  \caption{The decomposition of type I maps.}
  \label{fig:caseI}
\end{figure}

\subsection{Types II.1 and II.2 - Bridges}

To describe the corresponding decomposition for maps, we need to distinguish between type II.1 and type II.2. 

Consider first a labeled map $\rm$ of type II.1, i.e. a patch where the root edge is a bridge. The operation is the same as in  Section~\ref{sec:positive}: by deleting $e$, we decompose $\rm$ into 2 connected components $\rm_1$ and $\rm_2$, where $\rm_2$ is the component that contains $v_0$, and $\rm_1$ is the other one (and contains $v_1$), see Figure~\ref{fig:caseII12}. We reroot $\rm_2$ at the edge that follows $e$ in counterclockwise order around $v_0$, and $\rm_1$ at the edge that precedes $e$ in counterclockwise order around $v_1$.

Both $\rm_1$ and $\rm_2$ are clearly labeled quadrangulations with a $0-1$ boundary. Writing $2k_1$ and $2k_2$ for the degree of their root face, we have that $k=k_1+k_2+1$. 
\begin{figure}[ht]
  \centering 
  \includegraphics[page=6]{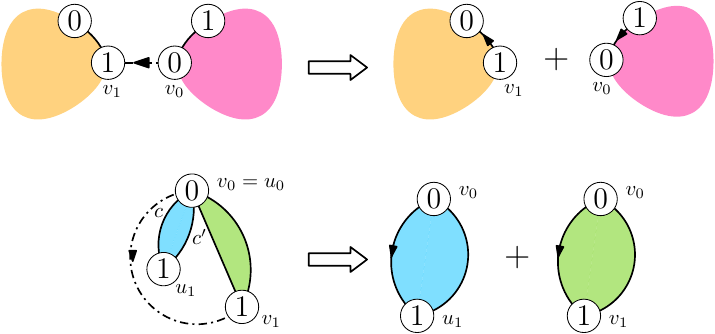}
  \caption{The decomposition of type II.1 (top) and II.2 (bottom) maps.}
  \label{fig:caseII12}
\end{figure}

The reverse operation consists in joining $\rm_1$ and $\rm_2$ with an edge between the root of $\rm_2$ and the co-root of $\rm_1$ and rooting the map on it.

\medskip

Now consider a labeled map $\rm$ of type II.2, i.e. a D-patch with a co-root face of degree 4 labeled $(0,1,0,1)$, and which is incident twice to the root vertex, and denote by $|k|$ its number of digons (here $k<0$). After deleting the root edge $e$, there are two corners of the root face that are incident to $v_0$. We write $c$ for the root corner and $c'$ for the other corner. We decompose $\rm$ by cutting through the vertex $v_0$ between the corners $c$ and $c'$, see Figure~\ref{fig:caseII12}. This creates two submaps $\rm_1$ and $\rm_2$, each containing one copy of $v_0$, and $\rm_1$ being the one that contains $v_1$. We root both of them at the cut corner of $v_0$. 

Each map is a D-patch and the digons of $\rm_1$ (resp.$\rm_2$) are the one that were between $c'$ and $c$ (resp. between $c$ and $c'$) around $v_0$, plus their root digon. Write respectively $|k_1|$ and $|k_2|$ for the number of digons of $\rm_1$ and $\rm_2$ (again with $k_1,k_2<0$), so that we have $|k|-1=(|k_1|-1)+ (|k_2|-1)$, and then $k = k_1+k_2+1$.  

The reverse operation consists in merging the root vertices of $\rm_1$ and $\rm_2$ and adding an edge from the former root corner of $\rm_2$ to the co-root vertex of $\rm_1$, enveloping $\rm_2$, and rooting the new map on the new edge.

\medskip

Both types are encoded into a well-charged tree by defining $\psi(\rm)$ as the tree with a binary root vertex with $\psi(\rm_1)$ as its left subtree and $\psi(\rm_2)$ as its right subtree. Assuming by induction that $c_{\psi(\rm_1)} = k_1$ and $c_{\psi(\rm_2)} = k_2$, we have $c_{\psi(\rm)} = c_{\psi(\rm_1)}+c_{\psi(\rm_2)}+1 = k_1+k_2+1 = k$. From the discussions above, this proves that the charge of $\psi(\rm)$ is equal to half the length of the outer face of $\rm$ (for Type II.1) or to minus the number of digons of $\rm$ (for Type II.2).

\subsection{Types II.3 and II.4 - Subpatch extraction and contraction}

The maps of the last types will be decomposed into a patch and a D-patch. To do so, we rely on an operation of \emph{subpatch extraction and contraction} used in \cite{MBM_AEP_20}, which we now describe.

\begin{defi}
    Let $\rm$ be a labeled map. A \emph{$\ell$-subpatch} of $\rm$ is a map obtained by taking a maximal connected component $C$ of vertices labeled $\ell$ and $\ell+1$ as well as all the edges and vertices that lay within its boundary. More formally, deleting $C$ partition the faces of $\rm$ into different connected components, and the $\ell$-subpatch associated to $C$ is obtained from $\rm$ by deleting the component of the root face. 
    
    If $v$ is a vertex of $\rm$ with label $\ell$, we denote respectively by $\Gamma_+(v)$ and $\Gamma_-(v)$ the $\ell$ and $(\ell-1)$ subpatches of $\rm$ containing $v$.
\end{defi}

\begin{figure}[ht]
  \centering 
  \includegraphics[scale=0.8, page=11]{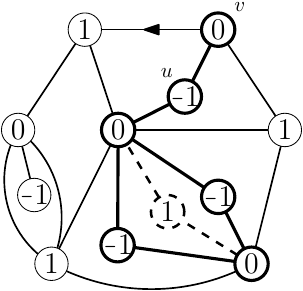}
  \caption{An example of a $(-1)$-patch. The connected component $C$ is represented in bold solid edges and the patch also includes the dashed edges. This patch is equal to $\Gamma_+(u)$ and $\Gamma_-(v)$.}
  \label{fig:ex patch}
\end{figure}

\begin{lemm}
\label{lem:subpatch}
	If $\rm$ is a quasi-quadrangulation, then any $\ell$-subpatch of $\rm$ is a patch up to shifting its labels by $-\ell$. The same holds if $\rm$ is a D-patch provided that the subpatch does not include the root vertex.
\end{lemm}

\begin{proof}
    Let $\rm$ be a quasi-quadrangulation and $\rp$ a $\ell$-subpatch of $\rm$. By construction, $\rp$ has a boundary labeled $\ell-(\ell+1)$, and all its inner faces are well labeled quadrangles since they were inner faces of $\rm$. Thus $\rp$ is a $\ell$-shifted patch. 
    
    If $\rm$ is a $D$-patch, then the same holds unless $\rp$ includes a digon. Since all digons are adjacent to the root vertex of $\rm$, this can happen only if $\rp$ contains the root vertex.
\end{proof}

The idea of the decomposition is to take a well chosen subpatch -which will be the patch in the decomposition- and contract it into a vertex. This contraction may create digons, so the resulting map will be a D-patch.

First, consider a patch with co-root face labeled $(0,1,2,1)$ (Type II.3.a). Denote by $u_1$ ($\neq v_1$) the vertex of the co-root face labeled $1$ and $u_0$ the one labeled $2$. We delete the root edge $e$ and define a patch and D-patch as follows (see Figure~\ref{fig:caseII3}). Let $\rp$ be the subpatch $\Gamma_+(v_0)$, rooted at $(v_0,u_1)$. Since $v_0$ is labeled $0$, $\rp$ is a patch according to Lemma~\ref{lem:subpatch}. 

Next, define $\tilde \rd$ as the submap of $\rm$ obtained as the connected component of $u_0$ after deleting all the edges in $\rp$, and rooted at $(u_1,u_0)$. Note that $\tilde \rd$ is not empty as it contains at least the two edges $\{u_0,u_1\}$ and $\{u_0,v_1\}$, and that the vertices which belong both to $\tilde \rd$ and to $\rp$ are necessarily labeled $1$. Finally, let $\rd$ be the map obtained by contracting $\Gamma_+(v_0)$
into $u_1^*$, rooting the map at the edge $u_1^*u_0$ and shifting all the labels by $-1$ (so that in particular $u_1^*$ is labeled $0$). Equivalently, $\rd$ is the map obtained from $\tilde \rd$ by merging together the vertices labeled $1$ on its outer boundary.

\begin{figure}[ht]
  \centering 
  \includegraphics[width=\textwidth, page=7]{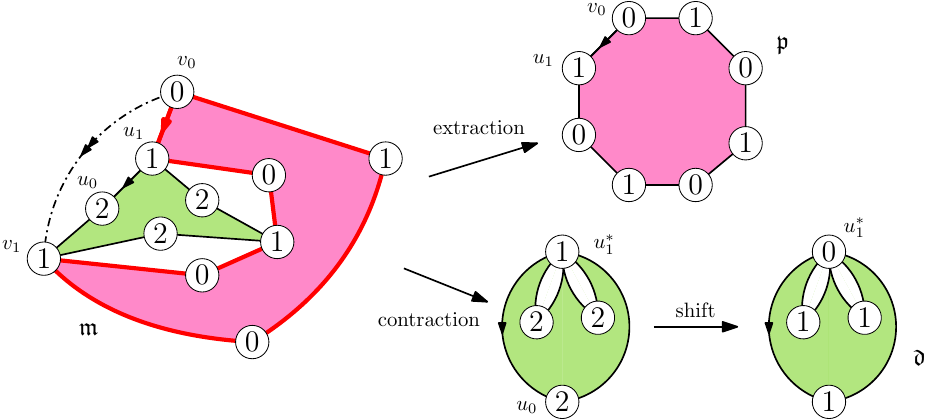}
  \caption{The decomposition of a type II.3.a map: subpatch extraction and contraction. Here $\ell(u_0) = 2$ so the patch (top) is $\Gamma_+(v_0)$.}
  \label{fig:caseII3}
\end{figure}

\begin{lemm}
\label{lem:patchDpatch}
	The maps $\rp$ and $\rd$ described above are respectively a patch and a D-patch. Furthermore, if $\rm$ and $\rp$ have respective outer degree $2k$ and $2i$ then $\rd$ has $j\coloneq i-k+1$ digons.
\end{lemm} 
\begin{proof}
	As observed in Lemma~\ref{lem:subpatch}, $\rp$ is a patch. 
	
    For $\rd$, it follows from the construction of $\rp$ that any edge incident to the outer face of $\tilde \rd$ was incident to a face labeled $(0,1,2,1)$ in $\rm$ (otherwise it would have been included in $\rp$). Hence, the vertices on the boundary of $\tilde \rd$ are labeled $1$ and $2$ alternately. Therefore, the new faces created in the contraction step are digons, and those are necessarily adjacent to the contracted vertex $u_1^*$, which is also the root vertex. Moreover, 
	by construction, any neighbour of $u_1^*$ had label $2$ before the shift, as otherwise it would have belonged to $\Gamma_+(v_0)$, and so they all have label $1$ in $\rd$. 
	
	Now consider a corner $c$ labeled $0$ incident to the outer face of $\rp$, there are $i$ of those. Either it was on the boundary on $\rm$, or it was in a face labeled $(0,1,2,1)$. In the second case, the two $1$ vertices were merged in $\rd$, creating a digon. Observe that all digons of $\rd$ except for its root face were created in that way. Finally, note that all corners of the root face of $\rm$ become corners on the root face of $\rp$ by definition of $\Gamma_+(v_0)$ (and since the vertices incident to the outer face of $\rm$ are labeled $0$ and $1$ alternately). Writing $j$ for the number of digons of $\rd$, this gives $i=k+j-1$, or equivalently $j=i-k+1$ as announced. 
\end{proof}

If $\rm$ is a patch with co-root face labeled $(0,-1,0,1)$, we do the same construction. The difference is that this time it is $v_0$ and not $u_1$ who belongs to both $\rp$ and $\tilde\rd$ and becomes the contracted vertex, and also the vertices on the boundary of $\tilde\rd$ are labeled $0$ and $-1$ alternately. So the root vertex of $\rd$ is labeled $0$ and all its neighbours are labeled $-1$, and to obtain a D-patch we shift all labels in $\rd$ except for the one of the root vertex by $+2$.

To get the tree encoding $\psi(\rm)$, we want to attach the two trees $\psi(\rp)$ and $\psi(\rd)$ to a binary vertex. All that remains is to choose which one is the left subtree and which one is the right subtree (i.e. in which case the left subtree has positive charge).
We choose to pick $\psi(\rp)$ as the right subtree when $u_0$ is labeled $2$, and as the left subtree when it is labeled $0$ to have the map with smaller labels on the right (another reason for this choice will appear in Appendix~\ref{app:explo}). Observe that by the second statement of Lemma~\ref{lem:contraction}, if $c_{\psi(\rp)}$ is half the length of the boundary of $\rp$ and $-c_{\psi(\rd)}$ is the number of digons of $\rd$, then $c_{\psi(\rm)} = i-j+1$ is half the length of the boundary of $\rm$, as desired.

\medskip

We now turn our attention to Type II.4. Consider a D-patch $\rm$ with a co-root face of degree $4$ adjacent only once to the root vertex. Delete the root edge and let $\rp$ be $\Gamma_-(u_0)$ if $\ell(u_0) = 0$ and $\Gamma_+(u_0)$ if $\ell(u_0) = 2$ and let $\rd$ be the map obtained by contracting $\rp$ and then merging $u_0$ with $v_0$ (see Figure~\ref{fig:caseII4}). 

Delete the root edge. Then, if $\ell(u_0) = 0$ let $\rp$ be $\Gamma_-(u_0)$ with all the labels shifted by $+1$, and rooted at the first edge which belongs to $\rp$ when turning clockwise around $u_0$ starting from $\{u_0,u_1\}$, see Figure~\ref{fig:caseII4}. Similarly, if $\ell(u_0) = 2$, let $\rp$ be $\Gamma_+(u_0)$ with all the labels shifted by $-2$, and rooted at the last edge which belongs to $\rp$ when turning clockwise around $u_0$ starting from $\{u_0,u_1\}$. In both cases, we let $\rd$ be the map obtained by contracting $\rp$ and then merging $u_0$ with $v_0$, and calling $v_0^*$ the resulting vertex.

\begin{figure}[ht]
  \centering 
  \includegraphics[width=\textwidth, page=8]{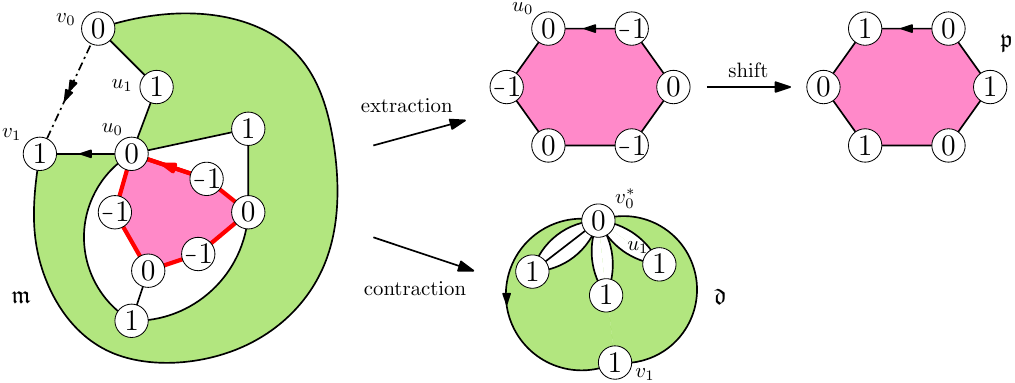}
  \caption{The decomposition of a map of type II.4.a: subpatch extraction and contraction. Here, we have $\ell(u_0)=0$, so the patch $\rp$ (top) is equal to $\Gamma_-(u_0)$ up to a shift of the labels.}
  \label{fig:caseII4}
\end{figure}

\begin{lemm}
\label{lem:contraction}
	The maps $\rp$ and $\rd$ described above are respectively a patch and a D-patch. Furthermore, if $\rm$ has $|k|$ digons and $\rp$ has a boundary of length $2i$ then $\rd$ has $j=i-k+1$ digons. 
\end{lemm} 
\begin{proof}
    Initially, all digons of $\rm$ are incident to $v_0$ by definition. Then, as in the proof of Lemma~\ref{lem:patchDpatch}, the contraction of $\rp$ creates only digons, that are incident to $u_0$. Moreover, after contraction $\rp$, both $u_0$ and $v_0$ have only neighbours labeled $1$ ($v_0$ by hypothesis on $\rm$ and $u_0$ by construction), so that the resulting map is a D-patch.

    Except for the digon created when merging $v_0$ and $u_0$, all digons of $\rd$ either were already in $\rm$ or were created by the contraction of $\rp$, and exactly $i$ were created this way. This gives that the number of digons in $\rd$ is equal to $-k+i+1$, as claimed. Note that there is a slight subtlety for the root face; the digon $(v_0,v_1)$ of $\rm$ is replaced by the digon $(v_0^*,v_1)$, which has no influence on the total number of digons. 
\end{proof}

It remains to decide how to encode this decomposition by a well-charged tree. As before, the tree encoding in obtained by attaching $\psi(\rp)$ and $\psi(\rd)$ to a binary vertex. Following the same logic as before, we choose to have $\psi(\rp)$ to be the right subtree when $\ell(u_0)=0$ and the left one when $\ell(u_0)=2$. Again, Lemma~\ref{lem:contraction} ensures that the charge behave as wanted. 

The following section is devoted to the description of the reverse operation, called the subpatch expansion.

\subsection{Types II.3 and II.4 - Subpatch expansion}

Observe that for both types we had $i+1-j = k \neq 0$, so we should take $j \neq i+1$. Fix $i,j>0$ with $j\neq i+1$, and let $\rp$ be a patch with a boundary of length $2i$ and $\rd$ be a D-patch with $j$ digons. Informally, the \emph{subpatch expansion} consists in opening the digons of $\rd$ by creating multiple copies of its root vertex and gluing them to $\rp$, see Figure~\ref{fig:caseIIrev}.

More precisely, we build the new map $\rm$ as follows:
\begin{enumerate}[1)]
	\item If $\rp$ corresponds to the left subtree of $\rt_\rm$, shift all its labels by $+2$ and denote by $c_1$ the root corner (which now has label $2$, otherwise shift the labels by $-1$ and denote by $c_1$ the corner of the root face which immediately follows the root corner (it has label $0$). List the following even labeled corners of the root face (in counterclockwise order) as $c_2, \ldots, c_{i+1}$ , with $c_{i+1}=c_1$.
	\item Open the inner digons of $\rd$ by creating a new copy of its root vertex for each inner digon. Denote by $w_1$ the one incident to the root edge and $w_2, \ldots w_j$ the others in \emph{clockwise} order.
	\item Then, for $1\leqslant s \leqslant \min(i+1,j)$, merge $w_s$ with the vertex incident to $c_s$ (in such a way that the edges incident to $w_s$ are inserted in $c_s$). In the case where both vertices do not have the same label, keep the label of $c_s$. 
	\item[4.a)] If $j>i+1$, some $w_s$ remain unmatched. Merge $w_{i+1}, \ldots, w_j$ together into a vertex $w^*$ and add a root edge from $w^*$ to the second $1$ on the boundary of the map so as to create an inner quadrangle.
	\item[4.b)] If $i\geqslant j$, some corners of $\rp$ remain unmatched. To finish the construction, we add a new root edge that encloses the co-root corner of $\rd$ into a quadrangle. To respect our rooting convention, if $\ell(c_{i+1})=0$, this edge is from the corner preceding $c_{i+1}$ to $c_j$ (and then is labeled $(-1,0)$). Otherwise, i.e. if $\ell(c_{i+1})=2$, the root edge is added from $c_{i+1}$ to the corner following $c_j$ (and in this case is labeled $(2,3)$).
	\item[4.c)] We forbid the case $j=i+1$ as a subpatch extraction cannot yield a pair of maps with these values. 
\end{enumerate}

\begin{figure}[ht]
  \centering 
  \includegraphics[width=\textwidth, page=10]{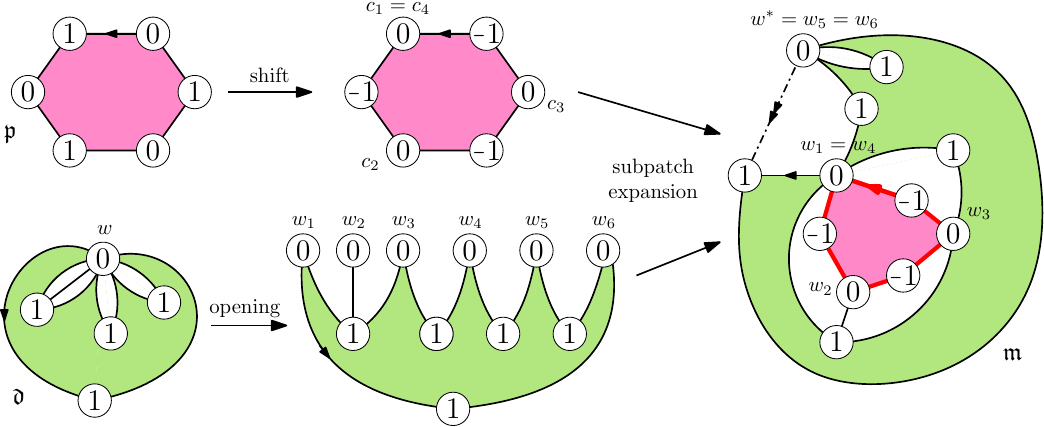}\\
  \vspace{0.2cm}
  \hrulefill
  \vspace{0.3cm}
  \includegraphics[width=\textwidth, page=9]{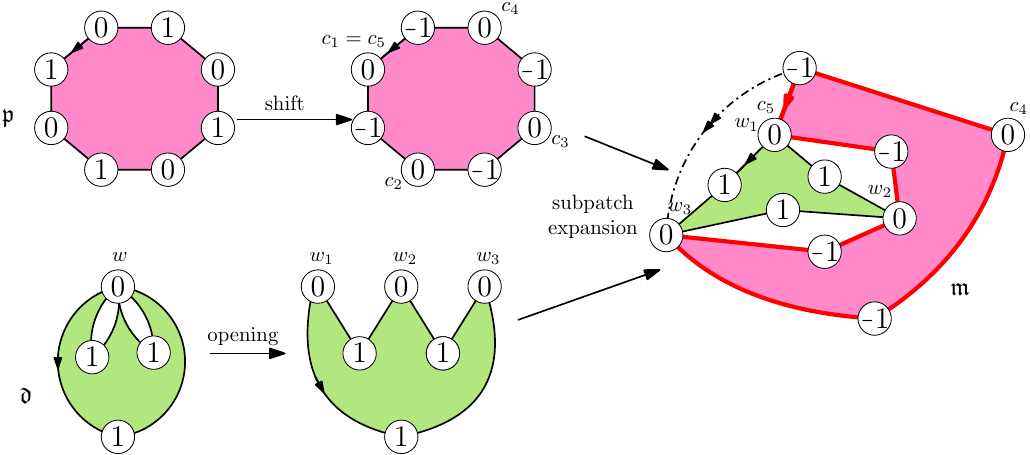}
  \caption{The reverse operation, subpatch expansion, for types II.4 (top) and II.3 (bottom).}
  \label{fig:caseIIrev}
\end{figure}

\begin{lemm}
\label{lem:expansion}
    The operation of subpatch expansion for a patch $\rp$ with boundary of length $2i$ and a $D$-patch $\rd$ with $j$ digons with $j \neq i+1$ gives a map $\rm$ that is :
    \begin{enumerate}[a)]
        \item A patch with boundary of length $2(i-j+1)$ if $i\geqslant j$.
        \item A $D$-patch with $j-i-1$ digons if $i<j-1$.
    \end{enumerate}
\end{lemm}
\begin{proof}
    Without loss of generality, we assume that we start the expansion by shifting the labels of $\rp$ by $-1$, the other case being symmetric.

    Denote by $w$ the root vertex of $\rd$. First, observe that when following the boundary of the ``opened'' D-patch $\tilde\rd$, two consecutive copies $w_s$ and $w_{s+1}$ of $w$ are separated by exactly one corner labeled $1$. Thus, since $w_s$ and $w_{s+1}$ are merged with vertices incident to two consecutive \emph{even} labeled corners of the root face of $\rp$, all the faces created have degree $4$. The resulting labeling is still valid since the merged vertices have the same label. (Note, that in the other case ($+2$ shift), the vertices incident to the corners are labeled $2$, but since all the neighbours of the $w_s$ are labeled $1$ the labeling remains valid.)
    
    If $j > i+1$, then before merging together the unmatched $w_s$ we have a quasi-\!\! quadrangulation whose boundary is made of the remaining opened digons of $\rd$, and so has labels alternating $0$ and $1$. Merging all the remaining copies reduces this boundary to a quadrangle, and might create some digons, which are all incident to $w^*$. Finally, the root edge added reduces the root face to a digon, while creating exactly one quadrangle. Vertices incident to $w^*$ form a subset of the vertices which were incident to $w$ in $\rd$, and hence they are all labeled $1$. This concludes the proof that $\rm$ is a $D$-patch in this case.
    
    If $i+1 > j$, then before adding the new edge we have a quasi-quadrangulation whose boundary is made of a part of the boundary of $\rp$ (with labels being alternately equal to $0$ and $-1$) and of the root digon of $\rd$, which has a corner $c$ labeled $1$. The edge added encloses $c$ in a quadrangle, leaving a boundary with only $0$ and $-1$, so that the resulting map is a $(-1)$-shifted patch.
    
    Finally, observe that merging a $w_s$ with a corner erases a digon of $\rd$ and reduces the length of the boundary of $\rp$ by $2$ (except for $w_1$). This means that the remaining map has ``charge" $i-j+1$, i.e. boundary of length $2(i-j+1)$ if it is a patch or $j-i-1$ digons if it is a $D$-patch.  
\end{proof}

\begin{theo}
    The subpatch contraction and subpatch expansion operations described above are inverse of each other. 
\end{theo}
\begin{proof}
    Lemmas~\ref{lem:contraction} and \ref{lem:expansion} guarantee that the maps obtained through both operations are of the expected type. It is clear that to reverse an expansion, one should contract $\rp$. Therefore, it suffices to observe that there is no other choice for gluing the maps together in such a way that all inner faces are quadrangles and which respects the rooting convention.
    
    We refer to \cite{MBM_AEP_20} for a more detailed proof in the case where the map obtained after the expansion is a D-patch. Their construction is slightly different: instead of shifting the labels of $\rp$, they apply the transformation $\ell \to 1-\ell$ but the arguments still hold. The same arguments apply similarly to the other cases.
\end{proof}

\section{Discussions}
\label{sec:discuss}

\subsection{Equation for patches}
\label{sub:equation}

In this section, we use our bijection to obtain a functional equation for the generating series of patches and D-patches. 

Let $P(t, x)$ be the generating function of patches counted by their number of edges (variable $t$) and the half length of their boundary (variable $x$) and let $D(t, x)$ be the generating function of D-patches counted by their number of edges (variable $t$) and their number of digons (variable $x$). 
We define $M(t, x) \coloneqq P(t, x) + D(t, 1/x)$, so that $M(t,x)$ is the generating function of both patches and D-patches. Informally, it amounts to considering D-patches as patches with a negative boundary length, by analogy with trees of negative charge. Then $M$ satisfies the following equation:
\begin{equation}
\label{eq:patches}
    M(t, x) = 1 + tx(M(t, x)^2 - x^{-1}[x^{-1}]M(t, x)^2) + \frac{t}{x}(M(t, x)-x[x^1]M(t, x)).
\end{equation}
The first term corresponds to the atomic map, the second term encompasses both bridges and subpatch expansions (types II) and the third term is adding an inner quadrangle (for patches) or digon (for D-patches) (types I), and we make sure to remove the terms that would lead to a map of ``boundary length $0$''.

Recall that deleting the non-root edge of the root face of a patch with boundary of length $2$ yields a labeled quadrangulation, and this transformation is bijective. So the generating function of labeled quadrangulation is $Q(t) = \frac1t[x^1]P(t,x) = \frac1t[x^1]M(t,x)$.

Bousquet-Mélou and Elvey Price already obtained this result in a completely computational way (with a different but equivalent equation on $M$). They then interpreted the equation in terms of trees and conjectured the bijection we provided.

\subsection{Boltzmann sampling}
\label{sub:boltzmann}

In the section, we explain how to efficiently generate patches according to the \emph{Boltzmann distribution}, extending the method proposed in \cite{MBM_AEP_25} for colorful patches, to general labeled patches. We refer to \cite{DFLS04} for more details on Boltzmann sampling.

Fix $t>0$ such that $M(t, 1)$ converges and denote by $P_k(t) \coloneq [x^k]P(t,x)$ and $D_k(t) \coloneq [x^k]D(t, x)$. For $k\geq 1$, we want to sample patches with a boundary of length $2k$ and such that a given patch $\rp \in \mathcal{P}^{(k)}$ is sampled with probability $t^{e(\rp)}/P_k(t)$. 

To do so, we start by deriving equations for $P_k$ and $D_k$ from ~\eqref{eq:patches}:
\[\left\{\begin{array}{l}
    P_0(t) = 1 \text{ and } D_0(t)=0\\
    P_k(t) = tP_{k+1}(t) + t\sum_{j=0}^{k-1}P_j(t)P_{k-1-j}(t) +2t\sum_{j>0} D_j(t)P_{k+j-1}(t),  \text{ for } k>0, \\
    D_k(t) = t\one_{k=1} +  tD_{k-1}(t) + t\sum_{j=1}^{k-1}D_j(t)D_{k-j+1}(t) +2t\sum_{j>0} P_j(t)D_{k+j-1}(t), \text{ for } k>0.
\end{array}\right .\]

We can now use them to recursively sample patches as follows:
\begin{enumerate}[a)]
    \item If $k=0$ then return the atomic map.
    \item Otherwise, choose which construction to apply by drawing a variable $i \in \{1, 2, 3\}$ with probabilities $\P(1) = \frac{tP_{k+1}(t)}{P_k(t)}$ and $\P(2) = \frac{t\sum_{j=0}^{k-1}P_j(t)P_{k-1-j}(t)}{P_k(t)}$.
    \item \begin{enumerate}[1)]
        \item If $i=1$, sample a patch in $\mathcal{P}^{(k+1)}$ and add a quadrangle (construction I.1).
        \item If $i=2$, draw $j \in \{0, \ldots k-1\}$ with probability $\P(j) \propto P_j(t)P_{k-1-j}(t)$. Then sample a patch in $\mathcal{P}^{(j)}$, a patch in $\mathcal{P}^{(k-1-j)}$ and join them (construction II.1).
        \item If $i=3$, draw $j \in \N$ with probability $\P(j) \propto D_j(t)P_{k+j-1}(t)$ and $\varepsilon \in \{-1, 2\}$ uniformly. Then draw a patch in $\mathcal{P}^{(k+j-1)}$ and a D-patch in $\mathcal{D}^{(j)}$ and perform a subpatch expansion with shift $\varepsilon$ (constructions II.3.a and II.3.b).
    \end{enumerate}
\end{enumerate}
D-patch are generated similarly:
\begin{enumerate}[a)]
    \item Choose which construction to apply by drawing a variable $i \in \{0, 1, 2, 3\}$ with probabilities $\P(0) = \frac{t\one_{k=1}}{D_k(t)}$, $\P(1) = \frac{tD_{k-1}(t)}{D_k(t)}$ and $\P(2) = \frac{t\sum_{j=0}^{k-1}D_j(t)D_{k-j+1}(t)}{D_k(t)}$.
    \item \begin{enumerate}
        \item[0)] If $i=0$ (and so $k=1$) return the map made of a single $(0,1)$ edge.
        \item[1)] If $i=1$, sample a D-patch in $\mathcal{D}^{(k-1)}$ and add a digon (construction I.2).
        \item[2)] If $i=2$, draw $j \in \{0, \ldots k-1\}$ with probability $\P(j) \propto D_j(t)D_{k-j+1}(t)$. Then sample a D-patch in $\mathcal{D}^{(j)}$, a D-patch in $\mathcal{D}^{(k-j+1)}$ and join them (construction II.2).
        \item[3)] If $i=3$, draw $j \in \N$ with probability $\P(j) \propto P_j(t)D_{k+j+1}(t)$ and $\varepsilon \in \{-1, 2\}$ uniformly. Then draw a patch in $\mathcal{P}^{(j)}$ and a D-patch in $\mathcal{D}^{(k+j+1)}$ and perform a subpatch expansion with shift $\varepsilon$ (constructions II.4.a and II.4.b).
    \end{enumerate}
\end{enumerate}

To be able to effectively use this algorithm, we need to be able to evaluate the series at a given point $t$. We refer to \cite{MBM_AEP_25} for discussion on this matter.

Equivalently, we could also first sample well-charged unary-binary trees of a given charge in the same fashion and then apply the bijection.

\subsection{Leaves and extrema}
\label{sub:leaves}

In \cite{MBM_AEP_25} Bousquet-Mélou and Elvey Price showed that the series of well-charged trees with charge $1$, counted by their number of vertices and their number of right leaves is equal to the one of labeled quadrangulation counted by their number of edges and their number of local minima. They also have computational indications that this remains true when considering the joint distribution of both the number of left and right leaves on one hand and both the number of local maxima and local minima on the other hand.

We strongly believe that the equality for the joint distribution is true not only for well-charged trees of charge one, but also for any well-charged trees with positive charge and patches. And, even for well-charged trees with negative charge and D-patches, with an adapted notion of local extrema (namely, where the root vertex and its neighbours do not count). Unfortunately, our recursive approach fails at preserving these statistics, 
even when considering only local minima/right leaves.

Let us be more precise. First, in the fully positive case, the
situation is much simpler: all labels of the (quasi-)quadrangulation are $0$ or $1$, right leaves are
mapped to vertices labeled~$0$, left leaves to vertices labeled~$1$,
and these vertices are precisely the local minima and local maxima,
respectively. We even have the stronger property that the extrema of a map are exactly those of the submaps it decomposes into.

However, this fails for the general bijection. More precisely, the
operations of types I.1, I.2, II.1 and II.2 preserve the labels in the
neighbourhood of every vertex, and hence preserve the status of a
vertex as a local extremum. The obstruction comes from the mixed-sign
cases, namely the subpatch expansion operations. Indeed, during a subpatch expansion with shift $-1$, while minima behave well, some maxima on the boundary of $\rp$ may not be maxima anymore in $\rm$ and some vertices that were adjacent to the root vertex of $\rd$ may become maxima in $\rm$.
With the shift $+2$, the symmetric issue
occurs for local minima. In either case, the recursive construction
does not keep track of the number of extrema created or destroyed.

Thus, although our bijection explains the unrefined enumeration, it
does not provide a bijective explanation of the refined distributions
observed in~\cite{MBM_AEP_25}. Constructing a bijection that preserves these
statistics remains a fascinating open problem. Such a bijection would
likely have to be more canonical than the recursive one presented here;
in particular, the desired statistic-preserving property may provide a
useful guide in the search for a genuinely non-recursive construction.

\appendix

\section{Non-recursive bijection in the fully positive case}
\label{app:explo}

In this section, we give a direct (i.e. non-recursive) description of the bijection of Section~\ref{sec:positive} between quadrangulations canonically labeled by $0$ and $1$ and fully positive trees.

In the general case, we also describe how to construct the well-charged tree directly on the map through an exploration process. Yet, in this situation, reconstructing the map from the resulting tree is rather cumbersome and does not shed much light on the structure. Another approach would be needed to obtain a genuinely non-recursive bijection.

\subsection{The trees as exploration paths - Fully positive case}

The rough idea to build the tree is the following: we put a vertex on each edge and a leaf on each vertex, then, starting from the root edge, we send particles that explore the faces of the map, adding edges corresponding to the path they follow (see Figure~\ref{fig:explo}). More precisely, the exploration is done as follows:

Initialize the current edge $e$ to be the root edge. We denote by $v_0$ and $v_1$ its extremities, with $\ell(v_i)=i$. While there are uncut edges, repeat the following:
\begin{itemize}
    \item[] Cross $e$ to discover a new face and cut $e$.
    \item[1)] If a new face is discovered (i.e. if $e$ is not a bridge), then go to the next edge around $v_0$ in the counterclockwise direction. 
    \item[2)] Otherwise, split the exploring particle: send one copy to the next edge around $v_0$ turning counterclockwise, and the other to the next edge around $v_1$ turning clockwise. If no such edge exists, attach a leaf to the corresponding vertex and terminate the exploration of that particle on it. Observe that each particle then evolves in a different connected component, so that the order in which they are explored does not matter.
\end{itemize}

\begin{figure}[ht]
    \centering
    \includegraphics[page=1, width=0.7\textwidth]{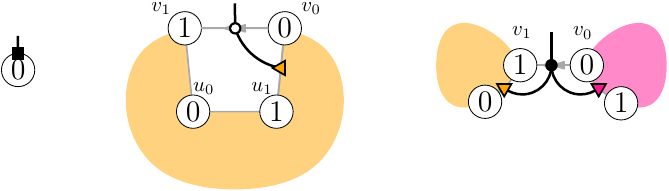}
    \caption{How to explore the quasi-quadrangulation to build the tree.}
    \label{fig:explo}
\end{figure}

It should be clear that this construction produces exactly the same tree as the recursive bijection: each vertex of the tree corresponds to the root edge of the submap it represents. Figure~\ref{fig:explo ex} provides a concrete example of the construction.

\begin{figure}[ht]
    \centering
    \includegraphics[page=2, width=0.8\textwidth]{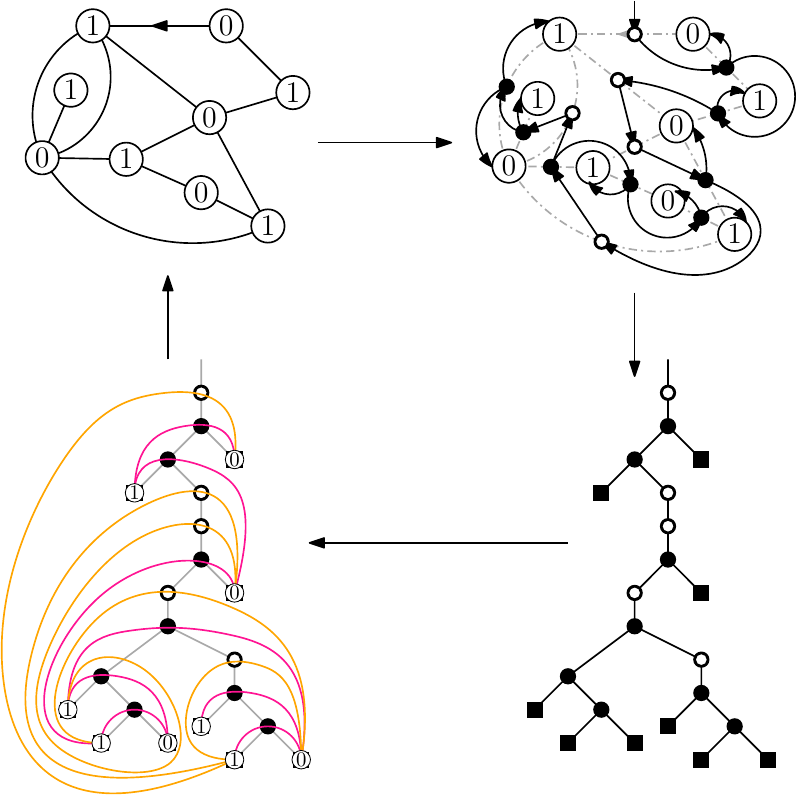}
    \caption{An example of how to built the tree by exploring the quasi-quadrangulation and how to recover the map.}
    \label{fig:explo ex}
\end{figure}

\subsection{Drawing the map on the tree}
The quasi-quadrangulation can also be recovered directly from the tree, and can even be ``drawn'' on top of it. To do so, we explore the tree from its leaves toward its root, and add one edge for each inner vertex (see Figure~\ref{fig:explo ex} for an example). More precisely, the procedure is as follows:

\begin{itemize}
    \item Convert each right (resp.left) leaf into a vertex labeled $0$ (resp. $1$). 
    \item Then, proceeding from the leaves toward the root, add one edge for each vertex $v$ as follows:
    \begin{itemize}
        \item If $v$ is binary, the new edge connects the first corner labeled $0$ encountered when performing a \emph{clockwise} contour of the connected component of the submap built so far for the subtree rooted at the right child of $v$, and the first corner labeled $1$ encountered when performing a \emph{counterclockwise} contour of the submap corresponding to the left child of $v$. 
        \item If $v$ is unary, the new edge connects the first corner labeled $0$ encountered when performing a clockwise contour of the connected component of the map built so far starting from $v$'s unique child, and the \emph{second} corner labeled $1$  when performing a counterclockwise contour also starting from $v$'s unique child. 
    \end{itemize}
\end{itemize}
It follows from the construction, that there is a unique way to add these edges, so that each edge crosses the tree exactly at its corresponding vertex, and such that no two added edges intersect. 

Again, it can be checked that this that this procedure recovers exactly the operation described in the recursive bijection. The fact that the construction always succeeds -- meaning that at each step the required vertices or corners labeled 0 and 1 indeed exist -- follows from the definition of well-charged trees, by the same arguments used in the recursive construction.

\subsection{The trees as exploration paths - General case}

In the general case, it is still possible to draw the trees directly on the map through an exploration of the faces as depicted in Figure~\ref{fig:explo general}. We still put a node per edge and a leaf per vertex, and the children of a node are still the root edges of the one or two submaps the map decomposes into. The subtlety is that during the exploration we need to remember whether we are in a patch or D-patch in order to be able to differentiate type I.1 from types II.2 and II.4.a and type II.3.a from type II.4.b. Another difficulty is that the digons will not be visible as we do not want to make the subpatch contraction. A solution is to, instead of contracting the patch, cut or color all the edges attached to it and remember those are supposed to be adjacent to the root vertex of the D-patch. Figure~\ref{fig:explo gen ex} provides a concrete example.

\begin{figure}[ht]
    \centering
    \includegraphics[page=3, width=\textwidth]{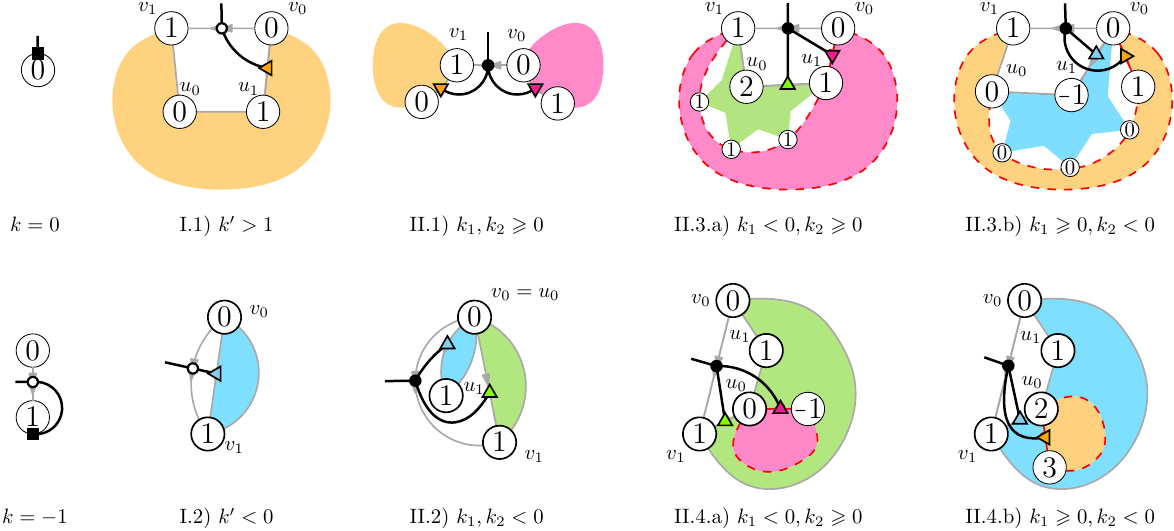}
    \caption{How to draw the trees by exploring the map in the general case.}
    \label{fig:explo general}
\end{figure}

\begin{figure}[ht]
    \centering
    \includegraphics[page=4, width=0.5\textwidth]{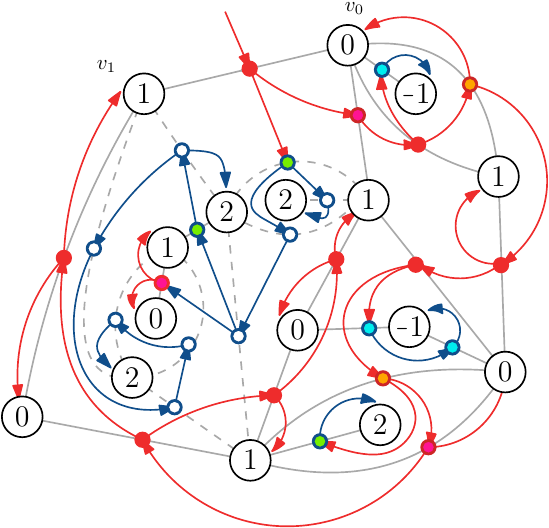}
    \caption{An example of the exploration in the general case. Red edges are created when exploring a patch and blue edges when exploring a D-patch. Dashed edges of the map are those that end up adjacent to the root vertex of a D-patch. Colors on the vertices match the color coding of the types in Figure~\ref{fig:explo general}.}
    \label{fig:explo gen ex}
\end{figure}

It is certainly possible to describe a process that builds the map directly on the tree, but the operations of subpatch expansion make it quite complicated. Again, digons should be replaced by dandling half edges, to avoid the ghost root vertex of D-patches.

\bibliographystyle{amsalpha}
\bibliography{EulerianOrientations}

\end{document}